\documentclass[10pt,reqno]{amsart}
\usepackage{tikz}
\usepackage{graphicx}
\usepackage{psfrag}
\usepackage{mathrsfs}
\usepackage{color}
\usepackage{amsmath,amssymb, amsfonts}

\usepackage{hyperref}
\usepackage[active]{srcltx}
\usepackage[normalem]{ulem}
\usepackage{lineno}

\numberwithin{equation}{section}

\usepackage{thmtools}
\usepackage{thm-restate}

\newtheorem{maintheorem}{Theorem}

\newtheorem{theorem}{Theorem}[section]

\newtheorem{lemma}[theorem]{Lemma}
\newtheorem{definition}[theorem]{Definition}

\newtheorem{remark}[theorem]{Remark}

\title[Unstable entropy and generic points]{Unstable entropy of partially hyperbolic diffeomorphisms along non-compact subsets}
\author{Gabriel Ponce}
\address{Departamento de Matem\'atica,
  IMECC-UNICAMP Campinas-SP, Brazil.}
  \email{gaponce@ime.unicamp.br}
\thanks{The author had the financial support of FAPESP process \# 2016/05384-0}
\date{}                                         

\begin{document}
\maketitle

\begin{abstract}
Given a partially hyperbolic diffeomorphism $f:M \rightarrow M$ defined on a compact Riemannian manifold $M$, in this paper we define the concept of unstable topological entropy of $f$ on a set $Y\subset M$ not necessarily compact. Using recent results of J. Yang \cite{Yang} and H. Hu, Y. Hua and W. Wu \cite{HuHuaWu} we extend a theorem of R. Bowen \cite{Bowen1973} proving that, for an ergodic $f$-invariant measure $\mu$, the unstable measure theoretical entropy of $f$ is upper bounded by the unstable topological entropy of $f$ on any set of positive $\mu$-measure. We define a notion of unstable topological entropy of $f$ using a Hausdorff dimension like characterization and we prove that this definition coincides with the definition of unstable topological entropy introduced in \cite{HuHuaWu}. \end{abstract}

\section{Introduction}

Given a smooth compact, connected Riemannian manifold $M$ without boundary we say that a $C^1$ diffeomorphism $f:M \rightarrow M$ is partially hyperbolic if for every point $x\in M$ there is a splitting 
\[T_xM = E^s(x)\oplus E^c(x) \oplus E^u(x) \]
and a Riemannian metric on $M$ such that for all unit vectors $v_s\in E^s(x), v_c\in E^c(x), v_u\in E^u(x)$ we have
\[||Df(x) \cdot v^s|| < ||Df(x)\cdot v^c || < ||Df(x)\cdot v^u ||,\]
and
\[ \max\{ ||Df(x) | E^s(x) || , ||Df(x)^{-1} | E^u(x) ||\} <1.\]
We call $E^s$ and $E^u$ the stable and unstable subbundles of $TM$ respectively. From results of \cite{HPS} there are $f$-invariant foliations $\mathcal F^{\tau}$ tangent to $E^{\tau}$, $\tau=s,u$, called the stable foliation (when $\tau=s$) and the unstable foliation of $f$ (when $\tau=u$).

From the definition we can say that a partially hyperbolic diffeomorphism is composed of a ``hyperbolic'' component, which is the dynamics induced by $f$ along the subbundles $E^s$ and $E^u$, and a ``central'' component which may or may not have contracting or expanding characteristics. A starting point to understand the dynamics of such diffeomorphisms is then try to understand how much influence does the hyperbolic part of $f$ exerts on the dynamics of $f$. A good example of such situation is the use of the accessibility property, which essentially says that any two points can be connected by a path tangent to the hyperbolic components of $f$, to obtain ergodicity for certain partially hyperbolic diffeomorphisms (see for example \cite{BW}).
In the seminal papers \cite{LY1,LY2} F. Ledrappier and L.S. Young were able to give a characterization of the metric entropy of a $C^2$ diffeomorphism $f$ in terms of the ``unstable characteristics'' of $f$, that is, in terms of the contribution of the unstable direction of $f$ to the entropy of $f$. In particular they were able to characterize the measures for which Pesin's entropy formula occurs. A central tool in their results is the concept of unstable metric entropy, which is defined via a certain conditional entropy using an increasing partition subordinated to the unstable foliation. Given the importance of such tool, some attention has been directed to the study of the so called unstable entropy of a partially hyperbolic dynamical system. For example, much more recently, J. Yang \cite{Yang} used such type of entropies to show that the set of Gibbs $u$-states of $C^{1+\alpha}$ partially hyperbolic diffeomorphisms is an upper semi-continuous function of the map in the $C^1$ topology and that the sets of partially hyperbolic diffeomorphisms with either mostly contracting or mostly expanding center are $C^1$ open. Another very interesting example of an application of such tools was given by J. Yang and A. Tahzibi in \cite{YangTahzibi} where they establish a beautiful criterium for $u$-invariance of an $f$-invariant measure based on how large is the entropy (compared to the unstable entropy of $f$) of the dynamics induced on the space of central leaves. M. Poletti \cite{Poletti2018} used this notions in relation to geometric growth to prove a $C^1$ conjugacy result for Anosov maps of the $3$-torus. Also very recently R. Saghin and J. Yang \cite{SaghinYang} used the entropy along expanding foliations as a tool to obtain Gibbs property of certain measures and then applied it to establish a local rigidity result of linear Anosov diffeomorphisms in terms of its Lyapunov exponents. 

With the idea of putting these ``unstable entropy tools'' altogether in a framework similar to the one already existing for the classical entropy, H. Hu, H. Hua and W. Wu \cite{HuHuaWu} redefined the concept of unstable metric entropy $h^u_{\mu}(f)$ (see Definition \ref{defi:umetric}), defined the concept of topological unstable entropy (see Definition \ref{utopological}) and proved several results in the direction of the classical theorems of entropy theory. For example they have proved that their definition of metric unstable entropy coincides with the definition given in \cite{LY1,LY2}, proved a version of the Shannon-McMillan-Breiman theorem for such entropies and also a variational principle relating the metric and topological unstable entropies.
While the unstable metric entropy is defined in \cite{LY1,LY2} through the conditional entropy $H_{\mu}(\xi | f\xi)$, where $\xi$ is an increasing partition subordinated to the unstable manifolds, the topological unstable entropy is defined by taking a definition via refinements of open covers, in analogy to \cite{Bowen}, restricted to a compact subset of the unstable manifold and is proved to be equal to the unstable volume growth of $f$ (see \cite{HuaSaghinXia}).

In \cite{Bowen1973} R. Bowen defined the topological entropy of a homeomorphism $f:X\rightarrow X$ on a subset $Y\subset X$ using a Hausdorff dimension like approach and proved that restricted to compact invariant sets such entropy coincides with the standard entropy. Furthermore it is also proved in \cite{Bowen1973} that, for an $f$-invariant measure $\mu$, the metric entropy of $f$ is upper bounded by the topological entropy restricted to any subset of $X$ with full measure. Inspired by these ideas we define the concept of $H$-unstable topological entropy along a non-compact subset $Y\subset M$ (see Definition \ref{defi:Hentropy}), which we denote by $h^u_H(f,Y)$, and we define the H-unstable topological entropy of a partially hyperbolic diffeomorphism $f$ as being the unstable entropy along the whole manifold $M$. We denote the H-unstable topological entropy of $f$ by $h^u_H(f)$. Then we extend Bowen's Theorem on the upper bound of the metric entropy to the context of unstable entropies (Theorem \ref{theo:main}.(2)) and, using these results and the unstable variational principle we show that the $H$-unstable topological entropy of $f$ coincides with the unstable topological entropy defined via open covers (Theorem \ref{theo:main}.(3)). This provides a characterization of the unstable topological entropy via a Hausdorff dimension approach. In the following theorem, for $x\in M$ and $Y\subset \mathcal F^u(x)$ a compact subset, $h^u(f,Y)$ denotes the unstable entropy of $f$ along $Y$ defined via open-covers (see Definition \ref{utopological}).

\begin{maintheorem} \label{theo:main} Let $f:M\rightarrow M$ be a $C^1$ partially hyperbolic diffeomorphism defined on a compact, connected Riemannian manifold $M$ without boundary. The following are true
\begin{itemize}
\item[1)]  for any $x\in M$, if $Y\subset \mathcal F^u(x)$ is a compact subset then $h^{u}_H(f,Y) \leq h^u(f,Y)$, 
\item[2)] if $\mu$ is an ergodic $f$-invariant probability measure then 
\[h^u_{\mu}(f) \leq h^u_H(f,Y)\]
for every measurable subset $Y\subset M$ with $\mu(Y)>0$; 
\item[3)]
\[h^{u}_H(f) = h^u_{top}(f).\]
\end{itemize}
\end{maintheorem}

\section{Preliminaries}

Along all the exposition $M$ is taken to be a smooth compact, connected Riemannian manifold without boundary, $f:M\rightarrow M$ is a $C^1$ (sometimes we require it to be $C^{1+\alpha}$) partially hyperbolic diffeomorphism and $\mathcal F^u$ and $\mathcal F^s$ are the unstable and stable foliations of $f$ respectively.

\subsection{Measure entropy for unstable foliation of partially hyperbolic diffeomorphisms}
In this section we recall the definition of unstable metric entropy of a partially hyperbolic diffeomorphism as defined in \cite{HuHuaWu} and state some properties which will be useful along the rest of the exposition. All along the paper $M$ will denote a smooth compact, connected Riemannian manifold without boundary.

Given a partition $\alpha$ of $M$ we will denote by $\alpha(x)$, $x\in M$, the element of $\alpha$ which contains $x$.

\begin{definition} \label{def:mensurable.partition}
We say that a partition $\alpha$ of $M$ is measurable with respect to $\mu$ if there exist a family $\{A_i\}_{i \in \mathbb N}$ of measurable sets and a measurable set $F$ of full $\mu$-measure such that 
if $B \in \alpha$, then there exists a sequence $\{B_i\}_{i\in \mathbb N}$, where $B_i \in \{A_i, A_i^c \}$ such that 
\[B \cap F = \bigcap_i B_i \cap F.\]
\end{definition}

For a certain $\varepsilon_0>0$ small enough denote by $\mathcal P = \mathcal P_{\varepsilon_0}$ the set of all finite measurable partitions of $M$ whose elements have diameter at most equal to $\varepsilon_0$. For each $\beta \in \mathcal P$ we can define a partition $\eta$ given by 
\[\eta(x) = \beta(x) \cap \mathcal F^u_{loc}(x)\]
where $\mathcal F^u_{loc}(x)$ denotes the local unstable manifold at $x$ whose size is greater than $\varepsilon_0$. This partition $\eta$ is then a measurable partition and $\eta$ is finer than $\beta$. Let $\mathcal P^u=\mathcal P^u_{\varepsilon_0}$ be the set of all partitions $\eta$ obtained in this manner.

\begin{definition}
A partition $\xi$ of $M$ is said to be subordinated to the unstable manifolds of $f$ with respect to a measure $\mu$ if for $\mu$-almost every $x$, $\xi(x) \subset \mathcal F^u(x)$ and $\xi(x)$ contains an open neighborhood of $x$ in $\mathcal F^u(x)$.
\end{definition}

Let $\mu$ be a probability measure on $M$ and $\alpha $ and $\eta$ two measurable partitions of $M$. The classical Rokhlin's Theorem (see \cite{Rohlin52}) guarantees the existence of a canonical system of conditional measures which disintegrates $\mu$ relative to $\eta$, that is, there exists a family of probability measures $\{\mu^{\eta}_x: x\in M\}$ such that 
\begin{itemize}
\item $\mu_x^{\eta}(\eta(x))=1$;
\item for every measurable subset $B\subset M$ the function $x\mapsto \mu_x^{\eta}(B)$ is a $\mathcal B(\eta)$-measurable function, where $\mathcal B(\eta)$ is the sub-$\sigma$-algebra generated by $\eta$, and;
\item for $B\subset M$ measurable,
\[\mu(B) = \int_M \mu^{\eta}_x(B) d\mu(x).\]
\end{itemize}
The conditional entropy of $\alpha$ given $\eta$ with respect to $\mu$ is defined by
\[H_{\mu}(\alpha | \eta) = -\int_M \log \mu_x^{\eta}(\alpha(x)) d\mu(x)\]
where $\{\mu^{\eta}_x : x\in M\}$ is a family of conditional measures of $\mu$ relative to $\eta$ as defined above.

\begin{definition} \label{defi:umetric}
The conditional entropy of $f$ with respect to a measurable partition $\alpha$ given $\eta \in \mathcal P^u$ is defined as 
\[h_{\mu}(f, \alpha|\eta) = \limsup_{n\rightarrow \infty} \frac{1}{n} H_{\mu}(\alpha_0^{n-1} | \eta),\]
where $\alpha_0^{m} := \bigvee_{i=0}^{m}f^{-i}\alpha$, for $m\in \mathbb N$.
The conditional entropy of $f$ given $\eta \in \mathcal P^u$ is defined by
\[h_{\mu}(f|\eta) = \sup_{\alpha \in \mathcal P}h_{\mu}(f,\alpha | \eta),\]
and the unstable metric entropy of $f$ is defined setting 
\[h_{\mu}^u(f) = \sup_{\eta \in \mathcal P^u} h_{\mu}(f | \eta).\]
\end{definition}

The following are standard properties in entropy theory so that we state them without a proof.
\begin{lemma}\label{lemma:auxx} \quad \\
\begin{itemize}
\item[a)] For $\beta, \eta$ measurable partitions and $n\in \mathbb N$ we have
\[h_{\mu}(f^n, \beta^{n-1}_0 | \eta) = n \cdot h_{\mu}(f,\beta | \eta).\]
\item[b)] For $\beta$ and $\eta$ measurable partitions we have
\[h_{\mu}(f,\beta | \eta) \leq h_{\mu}(f, \alpha | \eta) + H_{\mu}(\beta | \alpha).\]
\end{itemize}
\end{lemma}

\begin{remark} Given an $f$-invariant splitting $E^s\oplus E^c\oplus E^u$ of $M$, $f^{-1}$ is also partially hyperbolic with splitting $E^s_{f^{-1}} \oplus E^c \oplus E^u_{f^{-1}}$ where $E^s_{f^{-1}}:=E^u$ and $E^u_{f^{-1}}:=E^s$. However it is not true that, with respect to these splittings, $h^{u}_{\mu}(f)=h^{u}_{\mu}(f^{-1})$. A simple example is given by the following. Fix any $k_0\geq 5$ and take $f:\mathbb T^3 \rightarrow \mathbb T^3$ be the linear automorphism of $\mathbb T^3$ induced by 
\[A=\left(\begin{array}{lll} 0 & 0 & 1 \\ 0 & 1 & -1 \\ -1 & -1 & k_0 \end{array}\right).\]
An easy calculation (see \cite[Lemma $4.1$]{PT}) shows that $A$ has three real distinct eigenvalues $0<\lambda^s < \lambda^c < 1 < \lambda^u$. Write $E^{\tau}(x)$ the subspace of $T_x\mathbb T^3$ induced by the eigenspace of $A$ with respect to $\lambda^{\tau}$, $\tau=s,c,u$. As the Lebesgue measure $m$ on $\mathbb T^3$ is $f$ invariant and is clearly $u$-Gibbs we have by \cite[Theorem $3.4$]{Ledrappier} (also stated in \cite[Proposition $5.3$]{Yang}) that
\[h^u_{m}(f) = \int \log \operatorname{Jac}_f^u(x)dm(x) = \log \lambda^u.\]
 Now, if we regard $f^{-1}$ as a partially hyperbolic diffeomorphism with splitting $E^s_{f^{-1}}\oplus E^c_{f^{-1}} \oplus E^u_{f^{-1}}$ given by $E^s_{f^{-1}}(x):=E^u(x)$, $E^c_{f^{-1}}(x)=E^c(x)$ and $E^u_{f^{-1}}(x)=E^s(x)$ we have, by the same argument,
 \[h^u_{m}(f^{-1}) = \int \log \operatorname{Jac}_{f^{-1}}^u(x)dm(x) = -\log \lambda^s = \log \lambda^u + \log \lambda^c < h^u_{m}(f).\]
\end{remark}

\subsection{Unstable topological entropy of a compact subset}
Let $C^0_M$ denote the set of all finite open covers of $M$. Given $\mathcal U \in C^0_M$ denote $\mathcal U^n_m:= \bigvee_{i=m}^n f^{-i}\mathcal U$.
For any $K\subset M$ denote 
\[N(\mathcal U | K):= \min \{\text{card}( \mathcal V ): \mathcal V \subset \mathcal U, \bigcup_{V\in \mathcal V}V \supset K \},\]
and
\[H(\mathcal U | K):= \log N(\mathcal U | K),\]
where $\text{card}( \mathcal V )$ denotes the cardinality of the family of sets in $\mathcal V$.

\begin{definition}[\cite{HuHuaWu}] \label{utopological}
Let $d^u_x$ be the metric induced by the Riemannian structure on the unstable manifold $\mathcal F^u(x)$. With respect to the metric $d^u_x$ we denote by $\mathcal F^u(x,\delta)$ the open ball of radius $\delta$ inside $\mathcal F^u(x)$ centered at $x$. Given a compact subset $K\subset \mathcal F^u(x)$ we define the unstable entropy of $K$ by
\[h^{u}(f,K) = \sup_{\mathcal U \in C^0_M} \limsup_{n\rightarrow \infty} \frac{1}{n}H(\mathcal U_0^{n-1} | K).\]
If $Y\subset M$ is a compact subset, then we define 
\[h^u(f,Y)=\lim_{\delta \rightarrow 0} \sup_{x\in Y} h^{u}(f,Y\cap \overline{\mathcal F^u(x,\delta)}).\]
At last, the unstable topological entropy of $f$ is defined by:
\[h_{top}^u(f) = \lim_{\delta \rightarrow 0} \sup_{x\in M} h^{u}(f,\overline{\mathcal F^u(x,\delta)}).\]
\end{definition}

\begin{theorem}\cite[Theorem D]{HuHuaWu} \label{theo:variational}
Let $f:M\rightarrow M$ be a $C^1$-partially hyperbolic diffeomorphism. Then 
\[h^u_{top}(f) = \sup \{h^u_{\mu}(f) : \mu \in \mathcal M_f(M) \} = \sup \{h^u_{\mu}(f) : \mu \in \mathcal M_f^e(M) \}\]
where $M_f(M)$ (resp. $M_f^e(M)$) denotes the space of $f$-invariant probability measures (resp.  $f$-invariant ergodic probability measures) on $M$.
\end{theorem}

\section{Unstable topological entropy of non-compact subsets} \label{sec:defi}
Let $f:M \to M$ be a $C^1$ partially hyperbolic diffeomorphism defined on a compact Riemannian manifold $M$ and let $\mathcal F^u$ be its unstable foliation.

For each $x\in M$ we denote by $M_x$ the union of all the unstable leaves along the orbit of $x$, that is,
\[M_x:= \bigcup_{j=-\infty}^{\infty} \mathcal F^u(f^j(x)).\]
%
Given any open cover $\mathcal A$ of $M$ and a subset $E\subset M$ we say that $E$ \textit{is thinner than} $\mathcal A$, and we denote it by $E \prec \mathcal A$, if there exists a set $A \in \mathcal A$ such that $E\subset A$.
Fixed a finite open cover $\mathcal A$ of $M$, for each $E\subset M_x$ we denote: 
\[n_{f,\mathcal A}(E) = \text{ the biggest nonnegative integer for which } f^k(E) \prec \mathcal A \text{ for all integer } k\in [0,n_{f,\mathcal A}(E))\]
and \[D_{\mathcal A}(E) := e^{-n_{f,\mathcal A}(E)}.\]
If $\mathcal E = \{E_i : i=1,2,\ldots\}$ is a family of sets with $E_i \subset M_x$ for every $i\geq 1$ we define
\[D_{\mathcal A}(\mathcal E, \lambda) := \sum_{i=1}^{\infty} D_{\mathcal A}(E_i)^{\lambda}.\]
Similar to the idea of the definition of Hausdorff measure we define a measure $m^x_{\mathcal A, \lambda}$ in $M_x$ in the following way: for each $Y\subset M_x$
\[m^x_{\mathcal A, \lambda}(Y)= \lim_{\varepsilon \rightarrow 0} \; \inf \left\{ D_{\mathcal A}(\mathcal E, \lambda) : \mathcal E = \{E_i \subset M_x: i=1,2,\ldots \}, \bigcup_{i=1}^{\infty} E_i \supset Y , D_{\mathcal A}(E_i) < \varepsilon \right\}.\]
Now, in analogy to the definition of Hausdorff dimension, we define
\[h^{u}_{H,\mathcal A}(f,Y) = \inf \{ \lambda: m^x_{\mathcal A, \lambda}(Y)=0\}, \quad Y \subset M_x.\]

\begin{definition}\label{defi:Hentropy} For $Y\subset M$ and $x\in M$ we define
\[h^{u}_H(f,Y,x) = \sup_{\mathcal A}h^{u}_{H,\mathcal A}(f,Y \cap M_x)\]
where the sup ranges over all finite open covers of $M$.
Finally we define the $H$-unstable entropy of $Y$ by
\[h^{u}_H(f, Y ) = \sup_{x \in Y} h^u_H(f,Y,x).\]
We define the $H$-unstable entropy of $f$, and we denote it by $h^u_H(f)$, by taking $Y=M$, that is,
\[h^u_H(f) : = h^u_H(f,M).\]
\end{definition}
%

The following proposition is a generalization of properties which are well known to be satisfied for the classical entropy setting.
%
%

%

\begin{lemma}\label{lemma:prop2} For $f:M\rightarrow M$ a partially hyperbolic diffeomorphism of a compact manifold $M$ the following are true:
\begin{itemize}
\item[a)] $h^{u}_H(f,f(Y)) = h^{u}_H(f,Y)$, $Y\subset M$.
\item[b)] $h^u_H(f,\bigcup_{i=1}^{\infty} Y_i) = \sup_i h^u_H(f,Y_i)$, $Y_i\subset M$ for all $i$.
\item[c)] $h^u_H(f^m,Y) = m\cdot h^u_H(f,Y)$ for any natural number $m>0$, $Y\subset M$.
\end{itemize}
\end{lemma}
\begin{proof}
Given any $x\in Y$ and any finite open cover $\mathcal A$ of $M$, for any collection $\mathcal E=\{E_i : E_i \subset M_x\}_{i\in \mathbb N}$, with $\bigcup_{i=1}^{\infty}E_i \supset Y$, we have that the collection of the images $f(\mathcal E) := \{f(E_i) : f(E_i) \subset M_x\}_{i\in \mathbb N}$ satisfies $f(Y) \subset \bigcup_{i=1}^{\infty}f(E_i) $ and, since
\[n_{f,\mathcal A}(f(E_i)) = n_{f,\mathcal A}(E_i)-1,\]
we have
\[D_{\mathcal A}(f(\mathcal E), \lambda) = \sum_{i=1}^{\infty}D_{\mathcal A}(f(E_i))^{\lambda} = e^{\lambda}\cdot  \sum_{i=1}^{\infty}D_{\mathcal A}(E_i)^{\lambda}.\]
In particular if $m^{f(x)}_{\mathcal A, \lambda}(f(Y))=0$, then $m^x_{\mathcal A, \lambda}(Y)=0$ which implies 
\[h^{u}_{H,\mathcal A}(f,Y,x) \leq h^{u}_{H,\mathcal A}(f,f(Y),f(x)),\]
and consequently,
\begin{equation}\label{eq:rev0}
h^{u}_H(f,Y) \leq h^{u}_H(f,f(Y)).\end{equation}
The reverse inequality is obtained in a similar manner. Given any $x\in f(Y)$ and any finite open cover $\mathcal A$ of $M$, for any collection $\mathcal E=\{E_i : E_i \subset M_x\}_{i\in \mathbb N}$, with $\bigcup_{i=1}^{\infty}E_i \supset f(Y)$ and $D_{\mathcal A}(E_i)<1$, we have that the collection of the pre-images $f^{-1}(\mathcal E) := \{f^{-1}(E_i) : f^{-1}(E_i) \subset M_x\}_{i\in \mathbb N}$ satisfies $Y \subset \bigcup_{i=1}^{\infty}f^{-1}(E_i) $ and, since
\[n_{f,\mathcal A}(f^{-1}(E_i)) = n_{f,\mathcal A}(E_i)+1,\]
we have
\[D_{\mathcal A}(f^{-1}(\mathcal E), \lambda) = \sum_{i=1}^{\infty}D_{\mathcal A}(E_i)^{\lambda} = e^{-\lambda}\cdot  \sum_{i=1}^{\infty}D_{\mathcal A}(E_i)^{\lambda}.\]
In particular, if $m^{f^{-1}(x)}_{\mathcal A, \lambda}(Y)=0$ then $m^x_{\mathcal A, \lambda}(f(Y))=0$ which implies 
\[h^{u}_{H,\mathcal A}(f,f(Y),x) \leq h^{u}_{H,\mathcal A}(f,Y,f^{-1}(x)),\]
and therefore
\begin{equation}\label{eq:rev01}
h^{u}_H(f,f(Y)) \leq h^{u}_H(f,Y).\end{equation}
Thus, \eqref{eq:rev0} and \eqref{eq:rev01} proves the first item.

To prove the second item let us show the following subadditivity property \footnote{This property is actually true for a much more general set up, see for example Proposition $1.1$ from \cite{YP4}. Moreover the proof of the lemma we present here is based on the argument presented to prove the referred proposition in \cite{YP4}.}.
\begin{lemma}\label{lemma:revaux0}
For any finite open cover $\mathcal A$ of $M$ we have
\[m^x_{\mathcal A, \lambda}\left(\bigcup_{i=1}^{\infty} Y_i \cap M_x \right) \leq \sum_{i=1}^{\infty} m^x_{\mathcal A, \lambda}(Y_i \cap M_x), \quad x\in \bigcup_{i=1}^{\infty} Y_i .  \]
\end{lemma}
\begin{proof}
Given any $\delta>0$, $\varepsilon>0$ and $i\geq 0$ we can find, for each $i \in \mathbb N$, a constant $0<\varepsilon_i \leq \varepsilon$ and a countable collection $\mathcal E_i = \{E_{i,j}: E_{i,j}\subset M_x\}$ such that $\bigcup_{j=1}^{\infty}E_{i,j}\supset Y_i \cap M_x$ with $D_{\mathcal A}(E_{i,j}) \leq \varepsilon_i$, for every $j$, and with
\[\left| m^x_{\mathcal A, \lambda}(Y_i \cap M_x) - D_{\mathcal A}(\mathcal E_i,\lambda) \right| \leq \frac{\delta}{2^i}. \]
Now, the collection $\mathcal E = \{E_{i,j}: i\geq 0, j\geq 0 \}$ covers the union $\bigcup_{i=1}^{\infty}Y_i \cap M_x$ and $D_{\mathcal A}(E_{i,j}) \leq \varepsilon_i \leq \varepsilon$ for every $i,j \geq 0$. Furthermore we have
\begin{align*}
D_{\mathcal A}(\mathcal E,\lambda) & = \sum_{i,j}D_{\mathcal A}(E_{i,j})^{\lambda} \\
 & \leq \sum_{j} \sum_{i}\left( \frac{\delta}{2^i} + m^x_{\mathcal A, \lambda}(Y_i \cap M_x) \right) \\
 & = 2\delta+  \sum_{i}m^x_{\mathcal A, \lambda}(Y_i \cap M_x).\end{align*}
Since $\varepsilon$ can be taken to be arbitrarily small the last inequality implies
 \[m^x_{\mathcal A, \lambda}\left(\bigcup_{i=1}^{\infty} Y_i \cap M_x \right) \leq 2\delta+  \sum_{i}m^x_{\mathcal A, \lambda}(Y_i \cap M_x).\]
 But since $\delta$ can also be taken to be arbitrarily small the desired inequality follows.
 \end{proof}

Let us go back to the proof of the second item. Observe that for any $x\in \bigcup_{i=1}^{\infty}Y_i$ and any finite open cover $\mathcal A$ of $M$, if $\lambda > h^u_{H,\mathcal A}(Y_i)$ for every $i$ then, by definition, $m^x_{\mathcal A, \lambda}(Y_i\cap M_x)=0$ and by Lemma \ref{lemma:revaux0} we have
\[m^x_{\mathcal A, \lambda}\left(\left[\bigcup_{i=1}^{\infty} Y_i \right] \cap M_x \right) \leq \sum_{i=1}^{\infty}m^x_{\mathcal A, \lambda}(Y_i \cap M_x) = 0.\]
Thus
\[h^u_{H,\mathcal A}\left(f,\bigcup_{i=1}^{\infty} Y_i, x \right) \leq \lambda \Rightarrow h^u_{H,\mathcal A}\left(f,\bigcup_{i=1}^{\infty} Y_i, x \right) \leq \sup_{i}h^u_{H,\mathcal A}(Y_i). \]
Therefore, by taking the supremum over $x$ and over the finite cover $\mathcal A$, we have
\[ h^u_{H}\left(f,\bigcup_{i=1}^{\infty} Y_i \right) \leq \sup_{i}h^u_{H}(Y_i).\]
From the definition it is clear that if $E' \subset E$ then $m^x_{\mathcal A,\lambda}(E')\leq m^x_{\mathcal A,\lambda}(E)$, for any $x\in E'$ and any open cover $\mathcal A$ of $M$ since any collection covering $E$ also covers $E'$. Thus $h^u_H(f,E') \leq h^u_H(f,E)$. In particular, 
\[ \sup_{i}h^u_{H}(Y_i) \leq h^u_{H}\left(f,\bigcup_{i=1}^{\infty} Y_i \right),\]
concluding the proof of the second item of the proposition.

Let us prove the third item, finishing the prove of the proposition. Consider $\mathcal A$ to be a finite cover of $M$, $x\in M$ and $\mathcal E = \{E_i\}$ be a countable family of sets covering $Y\cap M_x$ and with $E_i \subset M_x$, for all $i$.  Let $m>0$ be a fixed natural number.
Observe that $(f^m)^k(E_i) \prec \mathcal A \Leftrightarrow m\cdot k<n_{f,\mathcal A}(E_i)$, that is, $n_{f^m,\mathcal A}(E_i) = \lfloor n_{f,\mathcal A}(E_i) /m \rfloor$. In particular
\[ \frac{n_{f,\mathcal A}(E_i)}{m}-1< n_{f^m,\mathcal A}(E_i) \leq \frac{n_{f,\mathcal A}(E_i)}{m},\]
and consequently
\[\sum_{i}e^{-n_{f,\mathcal A}(E_i)\frac{\lambda}{m}}\leq \sum_{i}e^{-n_{f^m,\mathcal A}(E_i)\lambda}<e^{\lambda}\cdot  \sum_{i}e^{-n_{f,\mathcal A}(E_i)\frac{\lambda}{m}}.\]
In particular 
\[m^x_{\mathcal A,\frac{\lambda}{m}, f} (Y) =0 \; \text{if and only if} \; m^x_{\mathcal A, \lambda, f^m}(Y)=0,\]
where the indexes $f$ and $f^m$ are put in the above expressions to indicate with respect to each function the evaluations of the definition of the $H$-unstable entropy are being made. Therefore,
\[h^u_{H,\mathcal A}(f^m,Y\cap M_x) = m\cdot h^u_{H,\mathcal A}(f,Y\cap M_x),\]
and by the arbitrariness of $\mathcal A$ and $x\in M$ it follows immediately that 
\[h^u_{H}(f^m,Y) = m\cdot h^u_{H}(f,Y),\]
as we wanted to show.

\end{proof}

Using Lemma \ref{lemma:prop2} we can give still another characterization \footnote{Some months after the submission of this paper X. Tian and W. Wu made a preprint \cite{TianWu} where they use the expression at the right hand side of Theorem \ref{prop:equiv} to prove several theorems concerning the unstable entropy of non-compact subsets. Although some of the results they obtained are already contained in this paper, their results where stablished following a different approach.} of the unstable entropy of a general subset. In the following Theorem, $h(f,Z)$ denotes the usual entropy of $f$ along a subset $Z\subset M$ as defined by Bowen in \cite{Bowen1973}.

\begin{theorem} \label{prop:equiv} For $Y\subset M$ we have
\[h^u_{H}(f,Y) = \lim_{\delta \rightarrow 0} \sup_{x\in M} h(f, \overline{\mathcal F^u(x,\delta)} \cap Y).\] 
\end{theorem}
\begin{proof}

First of all let us prove that 
\[h^u_{H}(f, \overline{\mathcal F^u(x,\delta)} \cap Y) = h(f, \overline{\mathcal F^u(x,\delta)} \cap Y),\]
for any $x\in M$ and $\delta>0$.

Let $\mathcal A$ be any open cover of $M$. For any cover $\widetilde{\mathcal E} = \{\widetilde{E_i} : i\in \mathbb N\}$ of $\overline{\mathcal F^u(x,\delta)} \cap Y$ we can consider the associate cover $\mathcal E = \{E_i \cap M_x : i\in \mathbb N\}$ which is subordinated to $M_x$. Also, $n_{f,\mathcal A}(E_i) \geq n_{f,\mathcal A}(\widetilde{E_i})$, thus
\begin{equation}\label{eq:ineqcompare}
m_{\mathcal A, \lambda}(\overline{\mathcal F^u(x,\delta)} \cap Y) \geq m^x_{\mathcal A, \lambda}(\overline{\mathcal F^u(x,\delta)} \cap Y).
\end{equation}
On the other hand, fixed $\mathcal A$ we clearly have
\[\left\{\mathcal E = \{E_i \subset M_x\}_i, \bigcup_{i=1}^{\infty} E_i \supset Z , D_{\mathcal A}(E_i) < \varepsilon \right\} \subset \left\{\mathcal E = \{E_i \}_i, \bigcup_{i=1}^{\infty} E_i \supset Z , D_{\mathcal A}(E_i) < \varepsilon \right\},\]
where $Z:= \overline{\mathcal F^u(x,\delta)} \cap Y$, thus we have the other hand of \eqref{eq:ineqcompare} from where we obtain
\[m_{\mathcal A,\lambda}(\overline{\mathcal F^u(x,\delta)} \cap Y) = m^x_{\mathcal A,\lambda}(\overline{\mathcal F^u(x,\delta)} \cap Y).\]
In particular we have
\[h(f,\overline{\mathcal F^u(x,\delta)} \cap Y) =  h^u_{H}(f,\overline{\mathcal F^u(x,\delta)} \cap Y).\]

Now, given any $x\in M$ and any $\delta>0$ we have $ \overline{\mathcal F^u(x,\delta)} \cap Y \subset Y$ thus $h^u_H(f, \overline{\mathcal F^u(x,\delta)} \cap Y) \leq h^u_H(f,Y)$ which implies
\[\lim_{\delta \rightarrow 0} \sup_{x\in M} h(f, \overline{\mathcal F^u(x,\delta)} \cap Y) \leq h^u_H(f,Y).\]

Let us prove the other side. Given any $\delta>0$, we can write $M_x$ as a countable union of the form
\[M_x = \bigcup_{i\in \mathbb N} \overline{\mathcal F^u(x_i,\delta)},\]
for certain points $x_i$, $i\in \mathbb N$. Thus for any $x\in M$, by item (b) of Lemma \ref{lemma:prop2} we have
\[h^u_{H,\mathcal A}(f,Y \cap M_x) \leq \sup_{i} h^u_{H,\mathcal A}(f,Y \cap \overline{\mathcal F^u(x_i,\delta)} )  \leq  \sup_{x\in M} h(f, \overline{\mathcal F^u(x,\delta)} \cap Y).\]
Making $\delta \rightarrow 0$ at the right side we obtain
\[h^u_{H,\mathcal A}(f,Y \cap M_x) \leq \lim_{\delta\rightarrow 0}  \sup_{x\in M} h(f, \overline{\mathcal F^u(x,\delta)} \cap Y).\]
Finally, taking the supremum over $x$ and over the finite cover $\mathcal A$ at the left side we obtain:
\[h^u_{H}(f,Y) = \lim_{\delta \rightarrow 0} \sup_{x\in M} h(f, \overline{\mathcal F^u(x,\delta)} \cap Y),\]
as desired.
\end{proof}

%
%
%

%
%
%
%
%
%

\section{Proof of Theorem \ref{theo:main}}

The arguments used to prove items $(1)$ and $(2)$ of Theorem \ref{theo:main} are similar to those given in \cite{Bowen1973} but, as the proof of the second item relies on the Shannon-McMillan-Breiman, which for the case of the unstable entropy, is given in terms of a conditional information function, we need to overcome the issue of always dealing with conditional measures instead of the original one. This is done in Lemma \ref{lemma:2}.
Item $(3)$ follows as a consequence of the two first items and the variational principle for unstable entropy.

\begin{proof}[Proof of item(1)]
The proof of the first item is similar to the first part of the proof of Proposition $1$ in \cite{Bowen1973}. However, we repeat the proof in details here for the sake of clarity.

Let $Y\subset \mathcal F^u(x)$ be a compact set and let $\mathcal A$ be any finite open cover of $M$. Let $\widetilde{\mathcal E_n}$ be a subcover of $Y$ with $N(\mathcal A_0^{n-1}|Y)$ members. Thus, if we consider $\mathcal E_n$ to be the collection of sets $E \cap \mathcal F^u(x)$ with $E\in \widetilde{\mathcal E_n}$ we have a family of $N(\mathcal A_0^{n-1} | Y)$ subsets of $\mathcal F^u(x)$ covering $Y$. Consequently
\[D_{\mathcal A}(\mathcal E_n,\lambda) \leq N(\mathcal A_0^{n-1}|Y)e^{-n\lambda}, \]
which implies
\[m^x_{\mathcal A,\lambda}(Y) \leq \lim_{n\rightarrow \infty} \left[ e^{-\lambda+ \frac{1}{n}H(\mathcal A_0^{n-1} | Y)} \right]^n.\]
If $\lambda > h^u(f,Y)$ then for $n$ large enough we have $-\lambda+ \frac{1}{n} H(\mathcal A_0^{n-1}|Y) <0$ which implies $m^x_{\mathcal A,\lambda}(Y) =0$ and consequently
\[h^{u}_{\mathcal A}(f,Y,x) \leq h^{u}(f,Y) \Rightarrow h^{u}_H(f,Y) \leq h^{u}(f,Y). \]
\end{proof}

The proof of the second item follows from the following lemmas.
\begin{lemma}\label{lemma:2}
Let $f:M\rightarrow M$ be a $C^1$ partially hyperbolic diffeomorphism defined on a compact, connected Riemannian manifold $M$ without boundary and let $\mu$ be an $f$-invariant ergodic measure. Assume that there exists a finite Borel partition $\alpha$ of $M$ such that every $x\in M$ is in the closure of at most $c$ sets of $\alpha$. Then, if $\mu(Y)>0$ we have
\[h_{\mu}^u(f)\leq h^u_H(f,Y,x_0) + \log c, \]
for $\mu$-almost every $x_0 \in Y$. In particular,
\[h_{\mu}^u(f)\leq h^u_H(f,Y) + \log c. \]
\end{lemma}
\begin{proof}
Let $\eta \in \mathcal P^u$ be any fixed partition. For each $y\in M$ consider
\[I_n(y):= I_{\mu}(\alpha_0^{n-1} | \eta)(y),\]
where $I_{\mu}(\xi | \beta)$ denotes the conditional information function of $\xi \in \mathcal P$ with respect to a measurable partition $\beta$ of $M$ defined by $I_{\mu}(\xi | \beta)(x):=-\log \mu_{x}^{\beta}(\xi(x))$.
By the Shannon-McMillan-Breiman Theorem for the unstable entropy \cite[Theorem $B$]{HuHuaWu} we have
\begin{equation}\label{eq:SMB}
\lim_{n\rightarrow +\infty}\frac{1}{n}I_n(x) = h_{\mu}^u(f) = h^u_{\mu}(f,\alpha|\eta) =:a
\end{equation}
for $\mu$-almost every point $x\in M$. Let $\tilde{Y}\subset Y$ be the subset of $Y$ for which \eqref{eq:SMB} occurs and take an arbitrary point $x_0\in \tilde{Y}$. For $\delta>0$ and $N\in \mathbb N$ denote
\[Y_{\delta,N}:=\left\{y\in \tilde{Y}: \frac{1}{n}I_n(y) \geq a-2\delta, \quad \forall n\geq N\right\}.\]
Observe that, for any fixed $m\in \mathbb  N$, $m>0$, we have
\[\tilde{Y} = \bigcup_{N \in \mathbb N}Y_{1/m,N}\; ,\]
so that we can take $N = N(m)\in \mathbb N$ for which $\mu(Y_{1/m,N}) > 0$.
Now, let $\mathcal B$ be a finite open cover of $M$ such that each set of $\mathcal B$ intersects at most $c$ elements of $\alpha$. Suppose that $\mathcal E = \{E_i\}_i$, $E_i \subset M_{x_0}$, covers $Y_0=Y\cap M_{x_0}$ with $D_{\mathcal B}(E_i) \leq e^{-N}$. 

If $\beta \in \alpha_{0}^{n_{f,\mathcal A}(E_i)} = \bigvee_{i=0}^{n_{f,\mathcal A}(E_i)} f^{-i}\alpha$ intersects $Y_{1/m, N}$, say $y_0 \in \beta \cap Y_{1/m, N}$, then $\beta = \alpha_{0}^{n_{f,\mathcal A}(E_i)}(y_0)$ and
\[-\frac{1}{n_{f,\mathcal A}(E_i) +1}\log \mu^{\eta}_{y_0}(\beta)  \geq a-\frac{2}{m}  \Rightarrow \mu^{\eta}_{y_0}(\beta) \leq \exp (-(a-2/m) n_{f,\mathcal A}(E_i)) .\]
Thus, given any $y_0 \in \beta \cap Y_{1/m, N}$ and any $y\in \eta(y_0)$ we have
\begin{equation}\label{eq:conds}
\mu^{\eta}_{y}(\beta) \leq \exp (-(a-2/m) n_{f,\mathcal A}(E_i)).
\end{equation}
Denote by $\eta(F)$ the $\eta$ saturation of a set $F\subset M$, that is, $\eta(F):= \bigcup_{y\in F}\eta(y)$. From the definition of the system of conditional measures we obtain
\begin{equation}\label{eq:conds2}
 \int_{M\setminus \eta(\beta \cap Y_{1/m,N})} \mu^{\eta}_{y}(\beta \cap Y_{1/m,N}) d\mu(y) =0.
\end{equation}
Now, by \eqref{eq:conds} and \eqref{eq:conds2} we have
\begin{align*}
\mu(\beta \cap Y_{1/m,N}) = & \int_{M}\mu^{\eta}_{y}(\beta \cap Y_{1/m,N}) d\mu(y) \\
 = & \int_{\eta(\beta \cap Y_{1/m,N})} \mu^{\eta}_y(\beta \cap Y_{1/m,N}) d\mu(y) + \int_{M\setminus \eta(\beta \cap Y_{1/m,N})} \mu^{\eta}_{y}(\beta \cap Y_{1/m,N}) d\mu(y) \\
 \leq & \exp (-(a-2/m) n_{f,\mathcal A}(E_i)).
%
\end{align*}
%
%
Thus, since $E_i\cap Y_{1/m,N}$ is covered by at most $c^{n_{f,\mathcal A}(E_i)+1}$ elements of the form $\beta \in \alpha_{0}^{n_{f,\mathcal A}(E_i)}$, we have
\begin{align*}
\mu(E_i \cap Y_{1/m,N}) \leq &  c^{n_{f,\mathcal A}(E_i)+1}\exp (-(a-2/m) n_{f,\mathcal A}(E_i)) \\
= & c\cdot \exp ((\log c-a+2/m) n_{f,\mathcal A}(E_i)).\end{align*}
Thus, for $\lambda = -\log c+a-2/m$ we have 
\[D_{\mathcal A}(\mathcal E,\lambda) = \sum_{i} \exp(-\lambda n_{f,\mathcal A}(E_i)) \geq \frac{1}{c}\sum_{i} \mu(E_i \cap Y_{1/m,N}) \geq \frac{1}{c}\mu(Y_{1/m,N}).\]
Letting the cover $\mathcal E$ vary we have $m^{x_0}_{\mathcal A,\lambda}(Y)\geq c^{-1}\mu(Y_{1/m,N}) >0$ which implies 
\[h^u_H(f,Y,x_0) \geq h^u_{H,\mathcal A}(f,Y \cap M_{x_0}) \geq \lambda = -\log c+a-2/m.\]
Taking $m\rightarrow \infty$ we obtain
\[h_{\mu}^u(f)\leq h^u_H(f,Y,x_0) + \log c\]
and 
\[h_{\mu}^u(f)\leq h^u_H(f,Y) + \log c. \]
as we wanted to show.
%
\end{proof}


To prove the second item we need the following Lemmas from \cite{Bowen1973}.

\begin{lemma}\cite[Lemma $2$]{Bowen1973}\label{lemma:22}
Let $\mathcal A$ be a finite open cover of $M$. For each $n>0$ there is a finite Borel partition $\alpha_n$ of $M$ such that $f^k \alpha_n \prec \mathcal A$ for all $k\in [0,n)$ and at most $n\cdot \text{card}( \mathcal A)$ sets in $\alpha_n$ can have a point in all their closures.
\end{lemma}

\begin{lemma}\cite[Lemma $3$]{Bowen1973} \label{lemma:3}
Given a finite Borel partition $\beta$ of $M$ and $\varepsilon >0$, there is an open cover $\mathcal A$ of $M$ so that $H_{\mu}(\beta | \alpha)<\varepsilon$ whenever $\alpha$ is a finite Borel partition of $M$ with $\alpha \prec \mathcal A$.
\end{lemma}

\begin{proof}[Proof of (2)] 
Let $\beta$ be a finite Borel partition of $M$ and $\varepsilon >0$. 
Let $\mathcal A$ be as in Lemma \ref{lemma:3} and $\alpha_n$ as in Lemma \ref{lemma:22}. Then, using the properties stated in Lemma \ref{lemma:auxx}, we have
\begin{align*}
h_{\mu}^u(f) = h_{\mu}(f,\beta | \eta) \stackrel{ \text{Lemma } \ref{lemma:auxx} \text{.(a)} }{=} & n^{-1}h_{\mu}(f^n, \beta_0^{n-1} | \eta ) \\
\stackrel{ \text{Lemma } \ref{lemma:auxx} \text{.(b)} }{\leq} &  n^{-1}h_{\mu}(f^n,\alpha_n | \eta) + n^{-1}H_{\mu}(\beta_0^{n-1} | \alpha_n) \\
\stackrel{ \text{Lemma } \ref{lemma:2} }{\leq }  & n^{-1} [ h^u_H(f^n,Y) + \log (n\cdot \text{card}(\mathcal A))]+n^{-1} \sum_{k=0}^{n-1} H_{\mu}(f^{-k}\beta | \alpha_n) \\
\leq & h^u_H(f,Y) + n^{-1} \log (n \cdot \text{card}(\mathcal A)) + n^{-1} \sum_{k=0}^{n-1} H_{\mu}(\beta | f^k \alpha_n) \\
\stackrel{ \text{Lemma } \ref{lemma:3} }{\leq }  & h^u_H(f,Y) + n^{-1} \log (n\cdot \text{card}(\mathcal A)) + \varepsilon .
\end{align*}
Above we have also used the classical fact from entropy theory $H_{\mu}(f^{-1}\beta | f^{-1}\alpha) = H_{\mu}(\beta | \alpha) $.
Taking $n\rightarrow \infty$ and $\varepsilon \rightarrow 0$ we get
\[h_{\mu}^u(f) \leq h^u_H(f,Y).\]
\end{proof}

\begin{proof}[Proof of (3)]
By (1) we have:
\[h^{u}_H(f,\overline{\mathcal F^u(x,\delta)}) \leq h^u_{top}(f, \overline{\mathcal F^u(x,\delta)}),\quad  \text{for any } \delta>0 .\]
Now we can write $M_x$ as a countable union of sets of the form $\overline{\mathcal F^u(x,\delta)}$ and, consequently, by item (b) of Lemma \ref{lemma:prop2} we conclude that
\[h^{u}_H(f) = h^{u}_H(f,M) \leq h^u_{top}(f).\]
On the other hand, by item (2) it follows that for any ergodic invariant measure $\mu$ we have $h^u_{\mu}(f) \leq h^u_H(f,M)$, thus
\[\sup_{\mu \in \mathcal M^e_f(M)} h^u_{\mu}(f) \leq h^u_H(f,M) = h^u_H(f).\]
Finally, by Theorem \ref{theo:variational} we conclude that $h^u_{top}(f) \leq h^{u}_H(f)$. Thus $h^{u}_H(f) = h^u_{top}(f)$ as we wanted to show.
\end{proof}

\section*{Acknowledgement}
The author thanks the anonymous referees for giving several constructive comments which substantially contributed to a better writing of this paper. We also thanks FAPESP for its financial support under process \# 2016/05384-0.

\bibliographystyle{plain}
\bibliography{Referencias2.bib}

\end{document}